\documentclass[11pt, reqno]{amsart}
\usepackage{amsmath,amsthm,amssymb,mathrsfs}
\usepackage[abbrev,non-sorted-cites]{amsrefs}
\usepackage{color,graphicx}
\usepackage{verbatim}
\usepackage{datetime}
\usepackage{hyperref}

\theoremstyle{plain}
    \newtheorem{thm}{Theorem}[section]
    \newtheorem{lem}[thm]{Lemma}
    \newtheorem{prop}[thm]{Proposition}
    \newtheorem{cor}[thm]{Corollary}

\theoremstyle{definition}
    \newtheorem{defn}[thm]{Definition}
    \newtheorem{assumption}[thm]{Assumption}
    \newtheorem{rmk}[thm]{Remark}
    \newtheorem*{ack*}{Acknowledgement}
    \newtheorem*{ques*}{Question}
    \newtheorem{const}[thm]{Construction}
    \newtheorem{eg}[thm]{Example}

\numberwithin{equation}{section}

\newcommand\ip[2]{\langle{#1},{#2}\rangle}

\newcommand\pl{\partial}

\newcommand\oh{\frac{1}{2}}
\newcommand\dd{{\mathrm d}}

\newcommand\w{\wedge}
\newcommand\sm{\sigma}
\newcommand\dt{\delta}

\newcommand\vep{\varepsilon}
\newcommand\vph{\varphi}
\newcommand\om{\omega}
\newcommand\ta{\theta}
\newcommand\gm{\gamma}
\newcommand\kp{\kappa}
\newcommand\af{\alpha}
\newcommand\bt{\beta}

\newcommand\ld{\lambda}

\newcommand\Om{\Omega}
\newcommand\Sm{\Sigma}
\newcommand\Gm{\Gamma}
\newcommand\Ld{\Lambda}

\newcommand\Dt{\Delta}

\newcommand\CA{\mathcal{A}}

\newcommand\CR{\mathcal{R}}

\newcommand\BS{\mathbb{S}}
\newcommand\BC{\mathbb{C}}
\newcommand\BP{\mathbb{P}}

\newcommand\BR{\mathbb{R}}

\newcommand\fA{\mathfrak{A}}
\newcommand\fB{\mathfrak{B}}
\newcommand\fC{\mathfrak{C}}
\newcommand\fD{\mathfrak{D}}
\newcommand\fE{\mathfrak{E}}
\newcommand\fF{\mathfrak{F}}

\newcommand\td{\tilde}

\newcommand\ot{\otimes}
\newcommand\op{\oplus}

\newcommand\ii{\sqrt{-1}}

\newcommand\ambn{\bar{\nabla}}

\newcommand\bddd{bounded in $(\infty,{\mathrm{loc}})$}

\DeclareMathOperator{\tr}{tr}

\DeclareMathOperator{\Hess}{Hess}
\DeclareMathOperator{\End}{End}

\DeclareMathOperator{\re}{Re}

\DeclareMathOperator{\vol}{Vol}

\DeclareMathOperator{\Ric}{Ric}

\DeclareMathOperator{\Sym}{Sym}

\begin{document}
\title[Stability of LMCF in K\"ahler-Einstein Manifolds with $c_1\leq0$]{Dynamical Stability of Minimal Lagrangians\\in K\"ahler-Einstein Manifolds\\of Non-Positive Curvature}
\author{Ping-Hung Lee}
\author{Chung-Jun Tsai}
\address{Department of Mathematics, National Taiwan University, and National Center for Theoretical Sciences, Math Division, Taipei 10617, Taiwan}
\email{b09201005@ntu.edu.tw, cjtsai@ntu.edu.tw}
\begin{abstract}
It is known that minimal Lagrangians in K\"ahler--Einstein manifolds of non-positive scalar curvature are linearly stable under Hamiltonian deformations.  We prove that they are also stable under the Lagrangian mean curvature flow, and therefore establish the equivalence between linear stability and dynamical stability.

Specifically, if one starts the mean curvature flow with a Lagrangian which is $C^1$-close and Hamiltonian isotopic to a minimal Lagrangian, the flow exists smoothly for all time, and converges to that minimal Lagrangian.  Due to the work of Neves \cite{Neves13}, this cannot be true for $C^0$-closeness.
\end{abstract}

\maketitle

\section{Introduction}
\subsection{Stability of Minimal Submanifold and Mean Curvature Flow}

Given a submanifold in a Riemannian manifold $F_0: \Sm \hookrightarrow (M,g)$, the mean curvature flow deforms it according to
\begin{align*}
    \frac{\pl F_t}{\pl t} &= H_{t} 
\end{align*}
where $H_t$ is the mean curvature vector of $F_t(L)$.
A natural question is the dynamical stability at a stationary state: if one starts the mean curvature flow with a submanifold close to a minimal one, will the flow exist for all time, and converge to a minimal submanifold?  As one may expect, it depends on the precise meaning of {closeness}.

Since the flow is the negative gradient flow of the volume functional, its second variation shall give some insights.  The Jacobi operator on a minimal submanifold $\Sm$ is $(\ambn^\perp)^*\ambn^\perp - (\CR + \CA)$, where $\CR$ is constructed from the curvature of $(M,g)$, and $\CA$ is a quadratic expression in the second fundamental form of $\Sm$.  It immediately implies that when $\CR+\CA\in\End(N\Sm)$ is negative definite everywhere, $\Sm$ is strictly stable in the linear sense.  Namely, $\dt^2_W\vol(\Sm) > 0$ for any non-trivial variation $W$. The negativity of $\CR+\CA$ is referred to as \emph{strong stability}.  In \cite{TsaiWang20}, the second-named author and M.-T.~Wang proved that a strongly stable minimal submanifold is dynamically stable in $C^1$.

\begin{thm}[{\cite{TsaiWang20}*{Theorem B}}] \label{thm_sstable}
    Let $\Sm^n\subset(M,g)$ be a compact, oriented, minimal submanifold which is strongly stable.  Then, if $\Gm$ is an $n$-dimensional submanifold which is close to $\Sm$ in $C^1$, the mean curvature flow $\Gm_t$ with $\Gm_0 = \Gm$ exists for all time, and $\Gm_t$ converges smoothly to $\Sm$ as $t\to\infty$.
\end{thm}
In \cite{LotaySchulze20}, Lotay and Schulze extended this theorem to a considerably weaker setting.

\subsection{Lagrangian Submanifold}

When the ambient manifold comes with additional structures, one can refine the discussion by considering submanifolds and variations with additional properties.
This paper concerns about the stability of minimal Lagrangian submanifolds.  Recall that a half-dimensional submanifold $L$ in a K\"ahler manifold $(M,g,J,\om)$ is called a Lagrangian submanifold if the restriction of $\om$ on $L$ vanishes.  Motivated by the study of moduli space of Calabi--Yau manifolds, minimal Lagrangian and Lagrangian mean curvature flow have received much attention.

Suppose that $L$ is a compact, oriented, minimal Lagrangian.  To incorporate the Lagrangian condition into the second variation of the volume, one has to understand the variation fields of nearby Lagrangians.  Note that a vector field $W$ normal to $L$ can be identified with the $1$-form $\xi_W = \om(W,\,\cdot\,)|_L$ on $L$.  It turns out that the infinitesimal condition of Lagrangian deformations is $\xi_W$ being $\dd$-closed.  Since harmonic $1$-forms are finite dimensional, one focuses on exact variations: $\xi_W = \dd u$ for $u\in C^\infty(L;\BR)$, which are called \emph{Hamiltonian variations}.

The second variation of the volume functional along Hamiltonian variations is derived by Oh \cites{Oh90, Oh93}.  When the ambient manifold $M$ is K\"ahler--Einstein $\Ric_g = \kp\,g$, the formula reads
\begin{align}
    \dt^2_{\dd u} \vol(L) &= \int_{L} \left( -\ip{\dd(\Dt_L u)}{\dd u} - \kp|\dd u|^2 \right) \dd V_L ~,
\label{Oh_formula} \end{align}
where $\Dt_L$ is the Laplace–Beltrami operator of the induced metric.  In particular, a minimal Lagrangian in a K\"ahler--Einstein manifold of non-positive scalar curvature is always stable under Hamiltonian variations.

On the other hand, if one starts the mean curvature flow with a Lagrangian submanifold, the Lagrangian condition needs not to be preserved along the flow.  As proved by Smoczyk \cite{Smoczyk96} and Oh \cite{Oh94}, the ambient space being K\"ahler--Einstein is a natural sufficient condition for the preservation of the Lagrangian condition.  For general symplectic manifolds, there are some attempts to modify the mean curvature flow in order to preserve the Lagrangian condition; see for example \cites{SmoczykWang11}.

\subsection{Dynamical Stability of Minimal Lagrangians}

Given a Hamiltonian variation $\dd u$, one can easily construct a nearby Lagrangian as follows.  The Weinstein theorem says that the tubular neighborhood of a Lagrangian $L$ in a symplectic manifold $(M,\om)$ is always modeled on the cotangent bundle of $L$ with the canonical symplectic form $\om_L$.  To be more precise, there exist an open neighborhood $U$ of $L$ in $T^*L$ and an embedding $\vph:U\to M$ such that $\vph^*\om = \om_L$ and the restriction of $\vph$ on the zero section is the given inclusion.  With the help of the Weinstein neighborhood, $(\vph\circ\dd u)(L)$ is a Lagrangian submanifold.  Since its position is given by $\dd u$, the induced metric is determined by the second order derivative $D^2u$.

We are now ready to state the main result of this paper.
\begin{thm}[Theorem \ref{thm_main}] \label{thm_main_intro}
    Suppose that $(M,g,J,\om)$ is a K\"ahler--Einstein manifold of non-positive scalar curvature.  Let $L\subset M$ be a compact, oriented, minimal Lagrangian, and fix a Weinstein neighborhood $\vph:U\subset T^*L\to M$.  Then, there exists $\vep > 0$ such that for any $u_0\in C^\infty(L;\BR)$ with $|u_0|+|\dd u_0| + |D^2 u_0| < \vep$, the Lagrangian mean curvature flow starting from $\vph\circ\dd u_0:L\hookrightarrow M$ exists for all time.  Moreover, as $t\to\infty$, the flow converges smoothly to $L$.
\end{thm}

The above norm is the static norm.  To say more, endow $L^n\subset M$ with the induced metric $\sm$.  The notation $D^2u$ means the Riemannian Hessian in $\sm$, and all the norms are taken by using $\sm$.  For the generalized Lagrangian mean curvature flow in the cotangent bundle of Smoczyk, Tsui and Wang \cite{SmoczykTsuiWang19}, a similar result was proved by Jin and Liu \cite{JL23}.

Here is the key ingredient of the proof.  An $(n,0)$-form is introduced in Construction \ref{const_holo_volume}.  It is not holomorphic when $\kp\neq0$, but still induces a well-behaved angle function on Lagrangians.  This angle function helps us to improve the technique of \cite{TsaiWang20} and to identify a monotone quantity along the flow when $\kp\leq0$, which is Proposition \ref{prop_test_psi}.

We make a comparison between Theorem \ref{thm_sstable} and \ref{thm_main_intro}.  As explained in \cite{TsaiWang20}*{section 3.2}, a minimal Lagrangian $L$ in a K\"ahler--Einstein manifold is strongly stable if $\Ric_{\sm} - \kp\sm$ is positive definite.  In particular, $(L,\sm)$ has positive Ricci curvature when $\kp = 0$, and thus has topological constraints.  For example, Theorem \ref{thm_main_intro} is applicable for a minimal Lagrangian subtorus in a flat torus, but Theorem \ref{thm_sstable} is not.  In a nutshell, Theorem \ref{thm_sstable} requires a stronger condition, while its conclusion works not only for Hamiltonian variations.

Note that adding a constant to $u_0$ leaves $\dd u_0$ unchanged.  By requiring $u_0$ to vanish at some point, the smallness of $|\dd u_0|$ implies the smallness of $|u_0|$.  This observation leads to the following rephrasing of Theorem \ref{thm_main_intro} in a more symplectic geometric flavor.
\begin{cor} \label{cor_hami}
    Suppose that $L$ is a compact, oriented, minimal Lagrangian in a K\"ahler--Einstein manifold of non-positive scalar curvature.  If one starts the mean curvature flow of a Lagrangian which is close to $L$ in $C^1$ and is Hamiltonian isotopic to $L$, then the flow exists for all time, and converges smoothly to $L$ as $t\to\infty$.
\end{cor}

Due to the work of Neves \cite{Neves13}, one may construct a Lagrangian $\td{L}$ Hamiltonian isotopic to $L$ which is arbitrarily close to $L$ in $C^0$, but the Lagrangian mean curvature flow starting at $\td{L}$ will develop a singularity in finite time.  Therefore, it is not possible to improve Corollary \ref{cor_hami} to $C^0$.  If one adopts a weaker notion of the mean curvature flow, it is possible to flow through the singularity for strongly stable $L$.  See the work \cite{LotaySchulze20}*{section 3.1} by Lotay and Schulze.

\begin{ack*}
    The authors are grateful to Wei-Bo Su, Valentino Tosatti, and Mao-Pei Tsui for helpful discussions.  The authors thank Mu-Tao Wang for suggestions on the earlier draft of this paper.  This research is supported in part by the Taiwan NSTC grants 112-2636-M-002-003 and 112-2628-M-002-004-MY4.
\end{ack*}

\section{Preliminaries}

\subsection{K\"ahler--Einstein metric}
Let $(M,g,J,\om)$ be a K\"ahler manifold of complex dimension $n$.  It is known (see for instance \cite{Kodaira06}*{ch.~3}) that the Levi-Civita connection naturally induces a connection on the canonical line bundle $K_M$, whose curvature form $\ii\rho$ is equivalent to the Ricci curvature of $g$.  Specifically,
\begin{align}
    \rho(X,Y) &= \Ric_g(JX,Y)
\label{Ricci_form} \end{align}
for any tangent vectors $X,Y$.  The metric is K\"ahler--Einstein if and only if
\begin{align}
    \rho &= \kp\,\om
\label{KE} \end{align}
for some constant $\kp\in\BR$.

\subsection{Lagrangian submanifold}
A half-dimensional submanifold $L\subset M$ is called a \emph{Lagrangian} submanifold if $\om|_L \equiv 0$.

Denote by $\tau_L$ the \emph{tautological $1$-fom} on $T^*L$.  A local coordinate $\{x^i\}_{i=1}^n$ for $L$ induces a coordinate $\{y_i\}_{i=1}^n$ for the fibers of $T^*L$ over the coordinate chart, and $\tau_L = \sum_{i=1}^n y_i\dd x^i$.  Its negative exterior derivative, $\om_L = -\dd\tau_L$, is a canonically defined symplectic form on $T^*L$.  

The Weinstein neighborhood theorem (\cite{McDuffSalamon17}*{ch.~3}) asserts that the tubular neighborhood of a Lagrangian submanifold is symplectomorphic to its cotangent bundle.  Such a parametrized neighborhood is not unique, and any choice suffices for the purpose of this paper.

\begin{thm} \label{thm_Weinstein_nbd}
    Let $L \subset (M,\om)$ be a Lagrangian submanifold.  Then, there exists an open neighborhood $U$ of the zero section in $T^*L$, and an embedding $\vph:U\hookrightarrow M$ such that the restriction of $\vph$ on the zero section is the given inclusion $L\subset M$, and $\vph^*\om = \om_L$.
\end{thm}
With the Weinstein neighborhood, a $1$-form $\xi$ on $L$ gives an embedding of $L$ in $T^*L$.  It is straightforward to verify that $\xi^*\tau_L = \xi$.  Thus, $\xi(L)$ is Lagrangian in $(T^*L,\om_L)$ if and only if $\dd\xi = 0$.  In particular, for any function $u$ on $L$, $\dd u(L)$ is always Lagrangian.

\subsection{Extrinsic geometry of Lagrangians}

Let $L \subset (M,g,J,\om)$ be a Lagrangian submanifold in a K\"ahler manifold.  One has the natural bundle isomorphisms:
\begin{align}
    NL\text{: normal bundle of $L$}  \stackrel{J}{\longrightarrow} TL \stackrel{\flat}{\longrightarrow} T^*L ~,
\label{bundle_iso} \end{align}
and they are parallel with respect to the connections induced by the Levi-Civita connection.  A direct computation shows that the composition of the maps in \eqref{bundle_iso} is $W\in NL \mapsto \om(W,\,\cdot\,)|_L$.
\begin{defn}
    Let $H$ be the {mean curvature vector} of a Lagrangian $L\subset M$ in a K\"ahler manifold.  Denote $\om(H,\,\cdot\,)|_L$ by $\af_L$, which is called the \emph{mean curvature $1$-form} of $L$.
\end{defn}

The mean curvature $1$-form is related to the complex structure of $M$ by the following lemma, which is proved in \cite{HarveyLawson82}*{sec.~III.2.D} and \cite{Oh94}*{Prop.~2.2}.
\begin{lem} \label{lem_mean_curvature}
     Suppose that $L$ is an oriented Lagrangian in a K\"ahler manifold $M$.  Consider the line bundle ${K_M}|_L$ with the connection induced by the Levi-Civita connection, which is denoted by $\nabla$.  Then, ${K_M}|_L$ admits a unique section $\Om_L$ whose restriction on the tangent space of $L$ coincides with the volume form $\dd V_L$ of $L$.  The section $\Om_L$ is of constant length, and
     \begin{align}
         \nabla\Om_L &= \ii\af_L\ot\Om_L
     \label{Lag_cpx_volumeform} \end{align}
\end{lem}
Suppose that $\{\sm^i\}_{i=1}^n$ is a local, oriented, orthonormal frame for $T^*L$.  The section $\Om_L$ is given by $(\sm^1+\ii J(\sm^1))\w\cdots\w(\sm^n+\ii J(\sm^n))$.

\subsection{Mean curvature 1-form of nearby Hamiltonian equivalent Lagrangians}

\begin{const} \label{const_holo_volume}
    Suppose that $L \subset (M,g,J,\om)$ is an oriented, \emph{minimal Lagrangian} submanifold in a \emph{K\"ahler--Einstein} manifold, $\Ric_g = \kp g$.  Fix a Weinstein neighborhood given by Theorem \ref{thm_Weinstein_nbd}, $\vph: U\to M$.
    
    For any $z\in \vph(U)$, $\gm_z(t) = \vph(t\cdot\vph^{-1}(z))$ for $t\in[0,1]$ is a curve connecting $z$ to $L$, which corresponds to a line segment in $T^*L$.  Take the section $\Om_L$ of ${K_M}|_L$ produced by Lemma \ref{lem_mean_curvature}.  Parallel transport $\Om_L$ along $\gm_z$ for any $z\in\vph(U)$ gives a smooth section of ${K_M}$ over ${\vph(U)}$, which will be denoted by $\Om_\vph$.  Note that $\Om_\vph$ is of constant length, and gives the exponential gauge for $K_M$ along the fibers of $T^*L$ (under $\vph$).

    Due to the minimality of $L$, \eqref{Lag_cpx_volumeform} says that $\Om_L$ is parallel along $L$.  Since $M$ is K\"ahler--Einstein, the curvature of $K_M$ is $\ii\kp\,\om$ \eqref{KE}.  By the calculation for exponential gauge in \cite{Uhlenbeck82}*{sec.~2}, one finds that on $\vph(U)$,
    \begin{align}
        \ambn\Om_\vph &= -\ii\kp\,\tau_L\ot\Om_\vph ~,
    \label{Weinstein_cpx_volumeform} \end{align}
    where $\ambn$ is the connection of $K_M$.  We omit $(\vph^{-1})^*$ on $\tau_L$ for simplicity.
\end{const}

\begin{rmk}
    Suppose that $M$ is a Calabi--Yau manifold with a holomorphic volume form $\Om$, and $L$ is calibrated by $\re\Om$.  The output $\Om_\vph$ of the above construction coincides with $\Om$.
\end{rmk}

For an oriented Lagrangian submanifold $\Gm\subset\vph(U)$, it follows from \cite{HarveyLawson82}*{sec.~III.1} that the restriction of $\Om_\vph$ on the tangent space of $\Gm$ is $e^{\ii\ta}\dd V_\Gm$ for some function $e^{\ii\ta}:\Gm\to \BS^1$.  It follows from Lemma \ref{lem_mean_curvature} that $\Om_\vph = e^{\ii\ta}\Om_\Gm$ on $\Gm$.  By \eqref{Lag_cpx_volumeform}, we find that on $\Gm$,
\begin{align*}
    \nabla\Om_\vph &= \nabla(e^{\ii\ta}\Om_\Gm) = \ii(\dd\ta + \af_\Gm)\ot\Om_\vph ~.
\end{align*}
Comparing it with the restriction of \eqref{Weinstein_cpx_volumeform} on $\Gm$ gives
\begin{align}
    \af_\Gm &= - \dd\ta - \kp\,{\tau_L}|_\Gm ~.
\label{mean_curvature_angle} \end{align}
For a function $u:L\to\BR$, applying \eqref{mean_curvature_angle} to $\Gm = \vph\circ\dd u$ gives the following lemma.

\begin{lem} \label{lem_mean_curvature_potential}
    Under the setting of Construction \ref{const_holo_volume}, suppose the $u:L\to\BR$ is a smooth function so that $\dd u(L)\subset U$.  Then, the mean curvature $1$-form of $F_u = \vph\circ\dd u$ is given by the exterior derivative of $-\ta(F_u) - \kp\,u$.  Here, $\ta(F_u)$ is the argument of $\Om_\vph/\dd V_{(\vph\circ\dd u)(L)}$, which is a well-defined smooth function.
\end{lem}

\subsection{Mean curvature flow in potential} \label{sec_LMCF_potential}

For a smooth function $u:L\times[0,t_0)\to\BR$, denote the spatial exterior derivative by $\dd u$.  Suppose that $F_u = \vph\circ\dd u$ evolves by the mean curvature flow; the non-parametric form of the equation reads
\begin{align}
    \left( \frac{\pl}{\pl t}F_u\right)^\perp &= H(t)
\label{MCF} \end{align}
where the right hand side means the mean curvature vector of $F_u(L)$ at time $t$.  According to the isomorphisms \eqref{bundle_iso}, Lemma \ref{lem_mean_curvature_potential} and the fact that $F_u(L)$ is Lagrangian, \eqref{MCF} becomes
\begin{align*}
    (\vph\circ\dd u)^*\left( \om\left( \frac{\pl}{\pl t}(\vph\circ\dd u) , \,\cdot\, \right) \right) &= -\dd(\ta(F_u) + \kp\,u) ~.
\end{align*}
Since $\vph^*\om = \om_L$, the left hand side can be computed explicitly in $T^*L$, and is equal to $-\dd(\frac{\pl}{\pl t} u)$.  Therefore, \eqref{MCF} in the current setting is equivalent to
\begin{align}
    \frac{\pl u}{\pl t} &= \ta(\dd u) + \kp\, u + C(t)
\label{LMCF_potential} \end{align}
for some $C(t):[0,t_0)\to\BR$, aka a time-dependent constant.  The function $C(t)$ has no effect on $\vph\circ\dd u$.  In this paper, $C(t)$ will be simply set to be $0$.

\subsection{Basic Riemannian geometry of the cotangent bundle} \label{sec_geom_cot}

Let $\sm = \sum_{i,j}\sm_{ij}(x)\dd x^i\ot\dd x^j$ be a Riemannian metric on $L$.  Denote by $D$ the Levi-Civita connection of $\sm$, and let $\Ld_{ij}^k$ be the Christoffel symbols of $D$.  The coordinate $\{x^i\}$ induces a coordinate $\{y_i\}$ for the fibers of $T^*L$.  The metric $\sm$ naturally gives a Riemannian metric on $T^*L$:
\begin{align}
    \td{\sm} &= \sum_{i,j}\sm_{ij}\dd x^i\ot\dd x^j + \sm^{ij}\eta_i\ot\eta_j
\label{linear_metric} \end{align}
where $\eta_i = \dd y_i - \sum_{j,k}\Ld_{ij}^k\,y_k\,\dd x^j$ and $\sm^{ij}$ is the inverse of $\sm_{ij}$.  The dual frame for $\{\dd x^i, \eta_i\}_{1\leq i\leq n}$ is $\{X_i,\frac{\pl}{\pl y_i}\}_{1\leq i\leq n}$ where
\begin{align}
    X_i &= \frac{\pl}{\pl x^i} + \sum_{j,k}\Ld^k_{ij}\,y_k\,\frac{\pl}{\pl y_j} ~.
\end{align}

For a smooth function $u:L\to\BR$, denote $(D^ku)(\frac{\pl\;}{\pl x^{i_k}},\cdots,\frac{\pl\;}{\pl x^{i_1}})$  by $u_{;i_1i_2\cdots i_k}$.  In particular,
\begin{align*}
    u_{;ij} = \frac{\pl^2 u}{\pl x^j\pl x^i} - \sum_k\Ld^k_{ji}\frac{\pl u}{\pl x^k} ~,~
    \text{and }~ \Delta_\sm u = \sum_{i,j}\sm^{ij}u_{;ij}
\end{align*}
is the Laplacian of $u$ with respect to the metric $\sm$.

For a function $u:L\to\BR$, let $F_u: L\to T^*L$ be the embedding given by $\dd u$.  In a local coordinate system, $F_u$ sends $(x^1,\ldots,x^n)$ to $(x^1,\ldots,x^n,\frac{\pl u}{\pl x^1},\ldots,\frac{\pl u}{\pl x^n})$.  Hence,
\begin{align}
    \frac{\pl F_u}{\pl x^i} &= \frac{\pl}{\pl x^i} + \sum_{j}\frac{\pl^2 u}{\pl x^j\pl x^i}\frac{\pl}{\pl y_j} = X_i + \sum_j u_{;ij}\frac{\pl}{\pl y_j} ~.
\label{1st_der_F} \end{align}

\section{Asymptotics of the Ambient Geometry}

For a Lagrangian $L$ in a K\"ahler manifold $(M,g,J,\om)$, fix a Weinstein neighborhood $\vph:U\subset T^*L \to M$.
Denote by $\ip{~}{~}$ the metric pairing on $U$ given by $\vph^*g$, and by $\ambn$ its Levi-Civita connection, which is equivalent to the Levi-Civita connection of $(M,g)$.

Let $\{g_{AB}(x,y)\}_{1\leq A,B\leq2n}$ be the coefficients of $\vph^*g$ in the coframe $\{\dd x^i, \eta_i\}_{1\leq i\leq n}$ introduced in section \ref{sec_geom_cot}; namely,
\begin{align} \label{ambient_metric}
    \vph^*g = \sum_{i,j} g_{ij}\,\dd x^i\ot\dd x^j + g_{i(n+j)}(\dd x^i\ot\eta_j+\eta_j\ot\dd x^i) + g_{(n+i)(n+j)}\,\eta_i\ot\eta_j ~,
\end{align}
and set $g_{AB} = g_{BA}$.  Equivalently,
\begin{align*}
    g_{ij} &= \ip{X_i}{X_j} ~,  & g_{i(n+j)} &= \ip{X_i}{\frac{\pl}{\pl y_j}} ~,  & g_{(n+i)(n+j)} &= \ip{\frac{\pl}{\pl y_i}}{\frac{\pl}{\pl y_j}} ~.
\end{align*}
Since $\vph^*\om = \om_L$, we abuse the notation and denote $\om_L$ by $\om$ from now on.  Similarly, denote $\vph^*J$ by $J$.  It follows from $\om = \sum_i\dd x^i\w\dd y_i$ and $\Ld_{ij}^k = \Ld_{ji}^k$ that
\begin{align*}
    \om({X_i},{X_j}) &= 0 ~,  & \om({X_i},{\frac{\pl}{\pl y_j}}) &= \dt_{ij} ~,  & \om({\frac{\pl}{\pl y_i}},{\frac{\pl}{\pl y_j}}) &= 0 ~.
\end{align*}

\begin{assumption} \label{assume_basic}
    Suppose that $L\subset(M,g,J,\om)$ is a compact, oriented, minimal Lagrangian in a K\"ahler--Einstein manifold.  Fix a Weinstein neighborhood $\vph:U\subset T^*L \to M$.  Let $\sm$ be the induced metric on $L$.  It induces an inner product on the fibers of $T^*L$: $|y|^2_\sm = |y_i\dd x^i|_\sm^2 = \sum_{i,j}\sm^{ij}(x)y_iy_j$.  Since $L$ is compact, there exists a positive constant $\vep_0 \leq 1$ such that $\{ y\in T^*L : |y|_\sm \leq \vep_0\}$ is contained in $U$.

    We also fix a finite coordinate covering of $L$ with the following significance: there exist constants $c_0, c_1, c_2, \ldots$ such that for any $x$ in each coordinate chart, and any $y$ with $|y|_\sm \leq \vep_0$,
    \begin{itemize}
        \item $\frac{1}{c_0} g_{AB}(x,0) \leq g_{AB}(x,y) \leq c_0\, g_{AB}(x,0)$ as $(2n)\times(2n)$ positive definite matrices;
        \item $\frac{1}{c_0} \sm_{ij}(x) \leq g_{ij}(x,y) \leq c_0\, \sm_{ij}(x)$ as $n\times n$ positive definite matrices;
        \item $|\pl_x^{(k)}\pl_y^{(\ell)} g_{AB}(x,y)| \leq c_{k+\ell}$ for every $k,\ell\geq0$.
    \end{itemize}
\end{assumption}

\begin{defn}
    Denote the above finite coordinate covering by $\{O_\bt\}$.  Any $x\in O_\bt$ and $Q\in\Sym_n(\BR)$ naturally define an element $\sum_{ij}Q_{ij}\,\dd x^i\ot\dd x^j$ in $\Sym^2(T_x^*L)$, whose norm is given by $|Q|_\sm^2 = \sum_{i,j,k,\ell}\sm^{ik}(x)\sm^{j\ell}(x)Q_{ij}Q_{k\ell}$.  Let
    \begin{align}
        \hat{O}_\bt &= \{(x,y,Q)\in O_\bt\times\BR^n\times\Sym_n(\BR) : |y|_\sm \leq \vep_0 , |Q|_\sm \leq \vep_0 \} ~.
    \end{align}
    A collection of smooth functions $f_\bt(x,y,Q)$ defined on an open neighborhood of $\hat{O}_\bt$ in $O_\bt\times\BR^n\times\Sym_n(\BR)$ is said to be \emph{\bddd} if
    \begin{align*}
        \left| \pl_x^{(i)}\pl_y^{(j)}\pl_Q^{(k)} f_\bt (x,y,Q) \right| \leq c_{i+j+k}
    \end{align*}
    for any $\bt$, any $i,j,k\geq0$ and any $(x,y,Q)\in\hat{O}_\bt$.
\end{defn}

\begin{eg} \label{bddd_example}
    It shall be clear that the following functions are defined on each coordinate chart, and the index $\bt$ will be omitted.
    \begin{enumerate}
        \item Due to Assumption \ref{assume_basic}, the Christoffel symbols $\Ld_{ij}^k(x)$ of $\sm_{ij}(x) = g_{ij}(x,0)$ are \bddd.  They only depend on $x$.
        \item The metric coefficients $g_{AB}(x,y)$ are \bddd, so are their Christoffel symbols.  They do not depend on $Q$.
        
        \noindent In particular, the functions $\ip{X_i}{\ambn_{X_j}X_k}$ and $\ip{X_i}{\ambn_{\frac{\pl}{\pl y^j}}X_k}$ are \bddd.
        \item Let
        \begin{align} \begin{split}
             \mu_{ij}(x,y,Q) &= g_{ij}(x,y) + \sum_{k}(g_{i(n+k)}(x,y)Q_{jk} + g_{j(n+k)}(x,y)Q_{ik}) \\
             &\qquad + \sum_{k,\ell}g_{(n+k)(n+\ell)}(x,y)Q_{ik}Q_{j\ell} ~.
        \end{split} \label{muij} \end{align}
        They are \bddd.  They are quadratic in the components of $Q$.
    \end{enumerate}
\end{eg}

The functions $\mu_{ij}$ defined by \eqref{muij} will be used quite often in the rest of this paper.  Their main properties are summarized in the following lemma.

\begin{lem} \label{lem_muij}
    Consider $\mu_{ij}(x,y,Q)$ defined by \eqref{muij}.  By taking $\vep_0$ smaller if necessary, $\mu_{ij}$ as a symmetric $n\times n$ matrix obeys
    \begin{align*}
        \frac{1}{2} \sm_{ij}(x) \leq \mu_{ij}(x,y,Q) \leq 2\,\sm_{ij}(x)
    \end{align*}
    for any $(x,y,Q)$ with $|y|_\sm\leq\vep_0$ and $|Q|_\sm\leq\vep_0$.  Moreover,
    \begin{align}
        \mu_{ij}(x,y,Q) - \sm_{ij}(x) &= \sum_{k} y_k\,\fA_{ij}^k(x,y) + \sum_{k,\ell} Q_{k\ell}\,\fB_{ij}^{k\ell}(x,y,Q)
    \end{align}
    for some $\fA_{ij}^k(x,y), \fB_{ij}^{k\ell}(x,y,Q)$ \bddd.
\end{lem}
\begin{proof}
    For the first assertion, note that
    \begin{align} \label{muij_dot}
        \mu_{ij}(x,y,Q) &= \ip{X_i + \sum_k Q_{ik}\frac{\pl}{\pl y_k}}{X_j + \sum_\ell Q_{j\ell}\frac{\pl}{\pl y_\ell}} \quad\text{at }(x,y) ~,
    \end{align}
    and is thus always positive definite.  Its eigenvalues with respect to $\sm_{ij}(x)$ are continuous functions on the tubular neighborhood of the zero section in $T^*L\op\Sym^2(T^*L)$.  Since $|y|_\sm\leq\vep_0$ and $|Q|_\sm\leq\vep_0$ defines a compact region, the first assertion follows.

    The second assertion follows directly from the expression \eqref{muij} and Taylor's theorem on $g_{AB}(x,y)$ in the $y$-variable.
\end{proof}

By using the condition that $L$ is a minimal Lagrangian, we deduce the following lemma.

\begin{lem} \label{lem_metric_asmp}
    Let $\mu^{ij}(x,y,Q)$ be the inverse of $\mu_{ij}(x,y,Q)$ \eqref{muij}.  Under Assumption \ref{assume_basic}, for any $k\in\{1,\ldots,n\}$,
    \begin{align*}
        \sum_{i,j}\mu^{ij}\cdot\om(\ambn_{X_i}X_j,X_k) = \sum_{\ell} y_\ell\,\fC^\ell_k(x,y,Q) + \sum_{\ell,m} Q_{\ell m}\,\fD^{\ell m}_k(x,y,Q) ~,
    \end{align*}
    where $\fC^\ell_k(x,y,Q), \fD^{\ell m}_{k}(x,y,Q)$ are functions {\bddd}.
\end{lem}
\begin{proof}
    To start, consider the $Q$-independent function $\sum_{i,j}\sm^{ij}\cdot\om(\ambn_{X_i}X_j,X_k)$.  When $y = 0$,
    \begin{align*}
        \sum_{i,j}\sm^{ij}\cdot\om(\ambn_{X_i}X_j,X_k) &= \sum_{i,j}\sm^{ij}\cdot\om(\ambn_{\frac{\pl}{\pl x^i}}\frac{\pl}{\pl x^j},\frac{\pl}{\pl x^k}) \\
        &= - \sum_{i,j}\sm^{ij}\left\langle{\ambn_{\frac{\pl}{\pl x^i}}\frac{\pl}{\pl x^j}}, {J(\frac{\pl}{\pl x^k})}\right\rangle \\
        &= - \left\langle{\sum_{i,j}\sm^{ij}\ambn^\perp_{\frac{\pl}{\pl x^i}}\frac{\pl}{\pl x^j}}, {J(\frac{\pl}{\pl x^k})}\right\rangle = 0 ~,
    \end{align*}
    where we have used the fact that $L$ is a minimal Lagrangian.  By writing
    \begin{align*}
        \sum_{i,j}\mu^{ij}\cdot\om(\ambn_{X_i}X_j,X_k) &= \sum_{i,j}(\mu^{ij}-\sm^{ij})\cdot\om(\ambn_{X_i}X_j,X_k) + \sum_{i,j}\sm^{ij}\cdot\om(\ambn_{X_i}X_j,X_k)
    \end{align*}
    and applying Lemma \ref{lem_muij}, this lemma follows.
\end{proof}

\section{Asymptotics of an Exact Graph}

Under Assumption \ref{assume_basic}, the differential of a function $u:L\to\BR$ with $|\dd u|_\sm\leq\vep_0$ defines a Lagrangian submanifold $F_u:L\to U\subset T^*L$.  The main purpose of this section is to derive asymptotics of the quantities related to $\Om_\vph$ given by Construction \ref{const_holo_volume}.  The constant $c$ in the estimates may change from line to line, but is always independent of $u$.

As in section \ref{sec_geom_cot}, denote by $D^2 u = \sum_{i,j}u_{;ij}\dd x^i\ot\dd x^j$ the Hessian of $u$ with respect to $\sm$.  By plugging \eqref{1st_der_F} into \eqref{ambient_metric}, the coefficients of the induced metric by $F_u$ are exactly $\mu_{ij}(x,\dd u,D^2u)$, with $\mu_{ij}$ defined by \eqref{muij}.  
It shall be clear that $\dd u$ means the vector $[u_{;i}]$, and $D^2u$ means the coefficient matrix $[u_{;k\ell}]$.

The first task is to establish the expansion of $\ta(F_u)$ given by Lemma \ref{lem_mean_curvature_potential}.  Note that $\ta(F_u)$ actually depends on $\dd u$ (the position) and $D^2u$ (the tangent space).
\begin{prop} \label{prop_angle_asmp}
    There exists a constant $c>0$ such that
    \begin{align*}
        \left| \ta(F_u) - \Dt_\sm u \right| &\leq c(|\dd u|_\sm^2 + |D^2 u|_\sm^2)
    \end{align*}
    for any smooth $u:L\to\BR$ with $|\dd u|_\sm\leq\vep_0$ and $|D^2u|_\sm\leq\vep_0$.
\end{prop}
\begin{proof}
    Denote by $\td{F}(x,s):L\times[0,1]\to T^*L$ the one-parameter family of Lagrangians given by $s\,\dd u$ for $s\in[0,1]$.  Since
    \begin{align} \label{tildeF_i}
        \td{F}_i &= \frac{\pl\td{F}}{\pl x^i} = X_i + s\sum_{j}u_{;ij}\frac{\pl}{\pl y_j} ~,
    \end{align}
    the induced metric by $\td{F}(x,s)$ for any fixed $s$ has coefficients $\mu_{ij}(x,s\dd u,sD^2u)$.  Write the variational field as
    \begin{align} \label{tildeF_s}
        \frac{\pl}{\pl s} &= \frac{\pl\td{F}}{\pl s} = \sum_i u_{;i}\frac{\pl}{\pl y_i} ~.
    \end{align}

    The partial derivative of $\ta(\td{F}(x,s))$ in $s$ can be computed by
    \begin{align}
        \frac{\pl}{\pl s} e^{\ii\ta} &= \frac{\pl}{\pl s}\left(\frac{\Om_\vph(\td{F}_1,\ldots,\td{F}_n)}{\sqrt{\det\mu}}\right) \notag \\
        \Rightarrow\quad \ii\frac{\pl\ta}{\pl s} &= - \oh \frac{\pl\log\det\mu}{\pl s} + \frac{e^{-\ii\ta}}{\sqrt{\det\mu}}\frac{\pl}{\pl s}\Om_\vph(\td{F}_1,\ldots,\td{F}_n) \label{theta_in_s}
    \end{align}
    
    Since $\Om_\vph$ is constructed by parallel transport along the radial curves of the fibers of $T^*L$ and $\{\td{F}(x,s)\}_{s\in[0,1]}$ is exactly the radial curve for every $x$, $\ambn_{\frac{\pl}{\pl s}}\Om_\vph = 0$.  It follows that
    \begin{align} \label{Om_in_s}
        \frac{\pl}{\pl s}\Om_\vph(\td{F}_1,\ldots,\td{F}_n) &= \Om_\vph(\ambn_{\frac{\pl}{\pl s}}\td{F}_1,\ldots,\td{F}_n) + \cdots + \Om_\vph(\td{F}_1,\ldots,\ambn_{\frac{\pl}{\pl s}}\td{F}_n) ~.
    \end{align}
    Because $\td{F}(L,s)$ is Lagrangian for every $s\in[0,1]$,
    \begin{align*}
        \ambn_{\frac{\pl}{\pl s}}\td{F}_i &= \sum_{j}\mu^{ij} \left( \ip{\ambn_{\frac{\pl}{\pl s}}\td{F}_i}{\td{F}_j}\td{F}_j + \ip{\ambn_{\frac{\pl}{\pl s}}\td{F}_i}{J(\td{F}_j)}J(\td{F}_j) \right) ~.
    \end{align*}
    Due to the fact that $\Om_\vph$ is an $(n,0)$-form, plugging this into \eqref{Om_in_s} gives
    \begin{align} \label{Om_in_s_a}
        \frac{\pl}{\pl s}\Om_\vph(\td{F}_1,\ldots,\td{F}_n) &= \sum_{i,j}\mu^{ij} \left( \ip{\ambn_{\frac{\pl}{\pl s}}\td{F}_i}{\td{F}_j} + \ii\ip{\ambn_{\frac{\pl}{\pl s}}\td{F}_i}{J(\td{F}_j)} \right)\Om_\vph(\td{F}_1,\ldots,\td{F}_n)
    \end{align}

    With \eqref{Om_in_s_a}, the real part of \eqref{theta_in_s} recovers the standard formula for $\frac{\pl}{\pl s}\log\det\mu$, and the imaginary part becomes
    \begin{align} \label{theta_in_s_a}
        \frac{\pl\ta}{\pl s} &= \sum_{i,j} \mu^{ij}\ip{\ambn_{\frac{\pl}{\pl s}}\td{F}_i}{J(\td{F}_j)} = \sum_{i,j} \mu^{ij}\,\om\left({\td{F}_j}, {\ambn_{\frac{\pl}{\pl s}}\td{F}_i}\right) ~.
    \end{align}
    Since $s\equiv 0$ corresponds to $L$ and the restriction of $\Om_\vph$ on the tangent space of $L$ is the volume form of $L$, $\ta$ vanishes when $s = 0$.  Therefore,
    \begin{align} \label{theta_Taylor}
        \ta|_{s=1} &= \left.\frac{\pl\ta}{\pl s}\right|_{s=0} + \int_0^1 (1-s)\frac{\pl^2\ta}{\pl s^2}\,\dd s ~.
    \end{align}

    {\textit{$1^{\text{st}}$ term on the right hand side of \eqref{theta_Taylor}}}.  By \eqref{tildeF_i} and \eqref{tildeF_s},
    \begin{align} \label{tildeF_is}
        \ambn_{\frac{\pl}{\pl s}}\td{F}_i &= \sum_k u_{;ik}\frac{\pl}{\pl y_k} + \sum_k u_{;k} \ambn_{\frac{\pl}{\pl y_k}}X_i + s\sum_{k,\ell} u_{;k}u_{;i\ell}\ambn_{\frac{\pl}{\pl y_k}}\frac{\pl}{\pl y_\ell} ~.
    \end{align}
    Together with \eqref{theta_in_s_a}, it gives
    \begin{align*}
        \left.\frac{\pl\ta}{\pl s}\right|_{s=0} &= \sum_{i,j,k} \sm^{ij}\,u_{;ik}\, \om(X_j, \frac{\pl}{\pl y_k}) + \sum_{i,j,k} u_{;k}\,\sm^{ij}\,\om( X_j, \ambn_{\frac{\pl}{\pl y_k}}X_i ) ~.
    \end{align*}
    The first term is $\Dt_\sm u$.  We \emph{claim} that $\om( \ambn_{\frac{\pl}{\pl y_k}}X_i, X_j )$ vanishes when $s=0$, and thus the second term is zero.  Note that
    \begin{align*}
        \ambn_{\frac{\pl}{\pl y_k}}X_i &= \ambn_{X_i}{\frac{\pl}{\pl y_k}} + [\frac{\pl}{\pl y_k}, \frac{\pl}{\pl x^i} + \sum_{j,\ell}\Ld^j_{i\ell}\,y_j\,\frac{\pl}{\pl y_\ell}] = \ambn_{X_i}{\frac{\pl}{\pl y_k}} + \sum_{\ell} \Ld_{\ell i}^k \frac{\pl}{\pl y_\ell} ~.
    \end{align*}
    It follows that
    \begin{align*}
        \om( X_j, \ambn_{\frac{\pl}{\pl y_k}}X_i ) &= \om( X_j, \ambn_{X_i}{\frac{\pl}{\pl y_k}} ) + \sum_\ell \Ld_{\ell i}^k\,\om(X_j, \frac{\pl}{\pl y_\ell}) \\
        &= - \om( \ambn_{X_i}X_j, {\frac{\pl}{\pl y_k}} ) + \Ld_{ji}^k ~,
    \end{align*}
    where the last equality uses the fact that $\ambn\om = 0$ and $\om(X_j, \frac{\pl}{\pl y_k}) = \dt_{jk}$.
    When $s=0$, $y = 0$,
    \begin{align*}
        \ambn_{X_i}X_j &= \sum_\ell\Ld_{ij}^\ell\frac{\pl}{\pl x^\ell} + \ambn^\perp_{\frac{\pl}{\pl x^i}}\frac{\pl}{\pl x^j} ~,
    \end{align*}
    and hence
    \begin{align*}
        \om( X_j, \ambn_{\frac{\pl}{\pl y_k}}X_i ) &= - \Ld_{ij}^k - \om( \ambn^\perp_{\frac{\pl}{\pl x^i}}\frac{\pl}{\pl x^j}, {\frac{\pl}{\pl y_k}} ) + \Ld_{ji}^k = 0 
        ~.
    \end{align*}
    This finishes the proof of the claim.

    {\textit{$2^{\text{nd}}$ term on the right hand side of \eqref{theta_Taylor}}}.  Since $\ambn\om = 0$, the derivative of \eqref{theta_in_s_a} in $s$ gives
    \begin{align} \begin{split}
        \frac{\pl^2\ta}{\pl s^2} &= -\sum_{i,j,k\ell} \mu^{ik}\frac{\pl\mu_{k\ell}}{\pl s}\mu^{j\ell}\,\om\left({\td{F}_j}, {\ambn_{\frac{\pl}{\pl s}}\td{F}_i}\right) \\
        &\quad + \sum_{i,j} \mu^{ij}\left[ \om\left( \ambn_{\frac{\pl}{\pl s}}{\td{F}_j}, {\ambn_{\frac{\pl}{\pl s}}\td{F}_i}\right) + \om\left({\td{F}_j}, {\ambn_{\frac{\pl}{\pl s}}\ambn_{\frac{\pl}{\pl s}}\td{F}_i}\right) \right] ~.
    \end{split} \label{theta_in_ss} \end{align}
    We compute the covariant derivative of \eqref{tildeF_is} in $s$:
    \begin{align} \begin{split}
        {\ambn_{\frac{\pl}{\pl s}}\ambn_{\frac{\pl}{\pl s}}\td{F}_i} &= \sum_{k,\ell} u_{;k}\,u_{;i\ell}\,\ambn_{\frac{\pl}{\pl y_k}}\frac{\pl}{\pl y_\ell} + \sum_{k,\ell} u_{;\ell}\,u_{;ik}\,\ambn_{\frac{\pl}{\pl y_\ell}}\frac{\pl}{\pl y_k} \\
        &\quad + \sum_{k,\ell} u_{;\ell}\,u_{;k}\, \ambn_{\frac{\pl}{\pl y_\ell}}\ambn_{\frac{\pl}{\pl y_k}}X_i + s\sum_{k,\ell,m} u_{;m}\,u_{;k}\,u_{;i\ell}\,\ambn_{\frac{\pl}{\pl y_m}}\ambn_{\frac{\pl}{\pl y_k}}\frac{\pl}{\pl y_\ell} ~.
    \end{split} \label{tildeF_iss} \end{align}
    With $|\dd u|_\sm\leq\vep_0$ and $|D^2u|_\sm\leq\vep_0$, it follows from \eqref{tildeF_i}, \eqref{tildeF_is} and \eqref{tildeF_iss} that
    \begin{align*}
        |\td{F}_i| &\leq c ~,  & |\ambn_{\frac{\pl}{\pl s}}\td{F}_i| &\leq c(|\dd u|_\sm + |D^2 u|_\sm) ~,  & |{\ambn_{\frac{\pl}{\pl s}}\ambn_{\frac{\pl}{\pl s}}\td{F}_i}| &\leq c(|\dd u|_\sm^2 + |D^2 u|_\sm^2)
    \end{align*}
    for some constant $c > 0$.  The norm is computed by $\vph^*g$ at $\td{F}(x,s)$.

    According to Lemma \ref{lem_muij}, the derivative of $\mu_{ij}$ in $s$ is
    \begin{align*}
        \frac{\pl\mu_{ij}}{\pl s}(x,s\dd u,sD^2u)
        &= \frac{\pl}{\pl s}\left(\sum_{k} su_{;k}\,\fA_{ij}^k(x,s\dd u) + \sum_{k,\ell} su_{;k\ell}\,\fB_{ij}^{k\ell}(x,s\dd u,sD^2u) \right) ~.
    \end{align*}
    Since $\fA_{ij}^k(x,y)$ and $\fB_{ij}^{k\ell}(x,y,Q)$ are \bddd, $|\frac{\pl\mu_{ij}}{\pl s}|$ is no greater than $c(|\dd u|_\sm + |D^2 u|_\sm)$ for some constant $c>0$.

    By applying these estimates to \eqref{theta_in_ss}, one finds that the absolute value of the last term in \eqref{theta_Taylor} is bounded from above by $c(|\dd u|_\sm^2 + |D^2 u|_\sm^2)$.  Note that $|\dd u|_\sm$ and $|D^2u|_\sm$ are assumed to be no greater than $\vep_0$, and only their lowest order terms matter.
\end{proof}

The next step is to study the (spatial) derivative of $\ta(F_u)$.
\begin{lem} \label{lem_theta_in_k}
    There exist $\fE^i_k(x,y,Q), \fF^{ij}_k(x,y,Q)$ which are {\bddd} such that
    \begin{align*}
        (\ta(F_u)+\kp u)_{;k} &= \sum_{i,j}\mu^{ij}u_{;jki} + \sum_{i}u_{;i}\,\fE^i_k(x,\dd u,D^2u) + \sum_{i,j}u_{;ij}\,\fF^{ij}_k(x,\dd u,D^2u)
    \end{align*}
    for any smooth $u:L\to\BR$ with $|\dd u|_\sm\leq\vep_0$ and $|D^2u|_\sm\leq\vep_0$.
\end{lem}
\begin{proof}
    As \eqref{tildeF_i}, denote $\frac{\pl F_u}{\pl x^i}$ by $F_i$.  By Lemma \ref{lem_mean_curvature_potential}, $-\dd(\ta+\kp u)$ is $(F_u)^*(\om(H(F_u),\,\cdot\,))$.  Since $F_u(L)$ is Lagrangian,
    \begin{align} \label{theta_in_k}
        (\ta+\kp u)_{;k} &= \sum_{i,j}\om(\mu^{ij}\ambn_{F_i}^\perp F_j,-F_k) = \sum_{i,j}\om(F_k, \mu^{ij}\ambn_{F_i} F_j) ~.
    \end{align}
    
    We compute
    \begin{align*}
        \ambn_{F_i}F_j &= \ambn_{F_i}\left( X_j + \sum_{k}u_{;jk}\frac{\pl}{\pl y_k}\right) \\
        &= \sum_k\frac{\pl u_{;jk}}{\pl x^i}\frac{\pl}{\pl y_k} + \ambn_{F_i} X_j + \sum_{k}u_{;jk}\ambn_{F_i}\frac{\pl}{\pl y_k} \\
        &= \sum_k u_{;jki} \frac{\pl}{\pl y_k} + \sum_{k,\ell}(\Ld_{ji}^\ell\,u_{;\ell k} + \Ld_{ki}^\ell\,u_{;j\ell})\frac{\pl}{\pl y_k} + \ambn_{X_i}X_j \\
        &\quad + \sum_k \left( u_{;ik}\ambn_{\frac{\pl}{\pl y_k}}X_j + u_{;jk}\ambn_{X_i}\frac{\pl}{\pl y_k} \right) + \sum_{k,\ell}u_{;jk}\;u_{;i\ell}\,\ambn_{\frac{\pl}{\pl y_\ell}}\frac{\pl}{\pl y_k} ~.
    \end{align*}
    It together with \eqref{theta_in_k} and Lemma \ref{lem_metric_asmp} finishes the proof of this lemma.
\end{proof}

By taking covariant derivatives, $D$ and $D^2$, on the expression in Lemma \ref{lem_theta_in_k}, it is not hard to prove the following corollary, and we omit its proof.
\begin{cor} \label{cor_d_theta_asmp}
    There exists $c > 0$ such that if $u:L\to\BR$ is a smooth function with $|\dd u|_\sm\leq\vep_0$ and $|D^2u|_\sm\leq\vep_0$, then
    \begin{enumerate}
        \item $\left|(\ta+\kp u)_{;k} - \sum_{i,j}\mu^{ij}u_{;jki}\right| \leq c(|D^2u|_\sm + |\dd u|_\sm)$;\smallskip
        \item $\left|(\ta+\kp u)_{;k\ell} - \sum_{i,j}\mu^{ij}u_{;jki\ell}\right| \leq c(|D^3u|^2_\sm + |D^3u|_\sm + |D^2u|_\sm + |\dd u|_\sm)$;\smallskip
        \item $\left|(\ta+\kp u)_{;k\ell m} - \sum_{i,j}\mu^{ij}u_{;jki\ell m} \right| \leq c(|D^4u|_\sm\cdot|D^3u|_\sm + |D^4u|_\sm + |D^3u|^3_\sm + |D^3u|_\sm + |D^2u|_\sm + |\dd u|_\sm)$;\smallskip
    \end{enumerate}
    for any $k,\ell,m$.
\end{cor}

For a function $u:L\to\BR$, let $\Gm_{ij}^k(F_u)$ be the Christoffel symbols of its Levi-Civita connection of $\mu_{ij}(x,\dd u, D^2u)$.  It is the restriction of $\ambn$ on the tangent space of $F_u$.  Recall that the difference between two connections is a tensor:
\begin{align}
    \Gm_{ij}^k(F_u) - \Ld_{ij}^k &= \oh\sum_\ell \mu^{k\ell}\left[ (\mu-\sm)_{j\ell;i} + (\mu-\sm)_{i\ell;j} - (\mu-\sm)_{ij;\ell} \right] \notag \\
    &= \oh\sum_\ell \mu^{k\ell}\left[ \mu_{j\ell;i} + \mu_{i\ell;j} - \mu_{ij;\ell} \right] ~, \label{connection_difference}
\end{align}
where the equality uses the fact that $D\sm = 0$.
\begin{lem} \label{lem_diffenece_Laplace}
    There exists a constant $c>0$ with the following property.  Let $u:L\to\BR$ be a smooth function with $|\dd u|_\sm\leq\vep_0$ and $|D^2u|_\sm\leq\vep_0$.  Then
    \begin{align*}
        \left| \Dt_\mu f - \tr_\mu(\Hess_\sm f) \right| &\leq c\,|\dd f|_\sm \cdot (|D^3u|_\sm + |D^2u|_\sm + |\dd u|_\sm)
    \end{align*}
    for any smooth function $f$ on $L$, where $\Dt_\mu$ is the Laplace operator of the metric $\mu_{ij}(x,\dd u, D^2 u)$.
\end{lem}
\begin{proof}
    We compute
    \begin{align*}
        \Dt_\mu f - \tr_\mu(\Hess_\sm f) &= \sum_{i,j} \mu^{ij}\sum_k\left(\Gm_{ij}^k(F_u) - \Ld_{ij}^k\right) \frac{\pl f}{\pl x^k} ~.
    \end{align*}
    By applying \eqref{connection_difference} and Lemma \ref{lem_muij}, this lemma follows.
\end{proof}

\begin{rmk}
    For the maximal principle argument of this paper, both $\Dt_\mu f$ and $\tr_\mu(\Hess_\sm f)$ will work.  We will go with the more geometric one, $\Dt_\mu f$.  If one uses $\tr_\mu(\Hess_\sm f)$, Lemma \ref{lem_diffenece_Laplace} will not be needed.  It seems that $\Dt_\sm$ is not a good choice.  This kind of issue also appears implicitly in \cite{SmoczykTsuiWang19}.
\end{rmk}

\section{Evolution Equations}

Now, suppose that $u:L\times[0,T)\to\BR$ is a smooth function such that the graph of $\dd u(\,\cdot\,t)$ evolves by the mean curvature flow.  As discussed in section \ref{sec_LMCF_potential}, we may assume that $u$ satisfies \eqref{LMCF_potential} with $C(t)\equiv0$.  As before, semicolon means the covariant derivative in $D$. 

\begin{lem} \label{lem_evo_eqn_u}
    There exists a constant $c>0$ with the following property.  Suppose that $u:L\times[0,T)\to\BR$ solves \eqref{LMCF_potential} with $C(t)\equiv0$, and $|\dd u|_\sm\leq\vep_0$, $|D^2 u|_\sm\leq\vep_0$ for all $t\in[0,T)$.  Then,
    \begin{align} \label{evo_eqn_u}
        (\frac{\pl}{\pl t} - \Dt_\mu)u^2 &\leq -|\dd u|_\sm^2 + 2\kp\,u^2 + c|u|(|D^3u|_\sm^2 + |D^2u|_\sm^2 + |\dd u|_\sm^2) ~.
    \end{align}
\end{lem}
\begin{proof}
    By \eqref{LMCF_potential},
    \begin{align*}
        (\frac{\pl}{\pl t} - \Dt_\mu)u^2 &= -2\sum_{i,j}\mu^{ij}u_{;i}u_{;j} + 2 u\cdot (\ta(F_u) - \Dt _\mu u) + 2\kp u^2 ~.
    \end{align*}
    Due to Lemma \ref{lem_muij}, $2\sum_{i,j}\mu^{ij}u_{;i}u_{;j} \geq |\dd u|_\sm^2$.  According to Proposition \ref{prop_angle_asmp}, Lemma \ref{lem_muij} and Lemma \ref{lem_diffenece_Laplace},
    \begin{align*}
        |\ta(F_u) - \Dt _\mu u| &\leq |\ta(F_u) - \Dt_\sm u| + |\Dt_\sm u - \tr_\mu(\Hess_\sm u)| + |\tr_\mu(\Hess_\sm u) - \Dt_\mu u| \\
        &\leq c'(|\dd u|_\sm^2 + |D^2 u|_\sm^2) + c''(|D^2u|_\sm + |\dd u|_\sm)|D^2u|_\sm \\
        &\quad + c'''|\dd u|_\sm(|D^3u|_\sm + |D^2u|_\sm + |\dd u|_\sm) ~.
    \end{align*}
    With the assumption that  $|\dd u|_\sm\leq\vep_0$, $|D^2 u|_\sm\leq\vep_0$ and some simple manipulation, it finishes the proof of this lemma.
\end{proof}

\begin{lem} \label{lem_evo_eqn_du}
    There exists a constant $c$ which has the following significance.  Suppose that $u:L\times[0,T)\to\BR$ solves \eqref{LMCF_potential}, and $|\dd u|_\sm\leq\vep_0$, $|D^2 u|_\sm\leq\vep_0$ for all $t\in[0,T)$.  Then,
    \begin{align}
        (\frac{\pl}{\pl t} - \Dt_\mu)|\dd u|_\sm^2 &\leq -\oh|D^2 u|_\sm^2 + c(|D^3u|_\sm\cdot|\dd u|_\sm + |\dd u|_\sm^2) \quad\text{and} \label{evo_eqn_du} \\
        (\frac{\pl}{\pl t} - \Dt_\mu)|D^2 u|_\sm^2 &\leq -\oh|D^3u|_\sm^2 + c(|D^3u|_\sm^2\cdot|D^2u|_\sm + |D^2u|_\sm^2 + |\dd u|_\sm^2) ~. \label{evo_eqn_ddu}
    \end{align}
\end{lem}
\begin{proof}
    {\textit{The equation \eqref{evo_eqn_du}}}.  Differentiating \eqref{LMCF_potential} gives
    \begin{align} \label{du_in_t}
        \frac{\pl}{\pl t}u_{;k} &= (\ta + \kp u)_{;k} ~.
    \end{align}
    We compute
    \begin{align} \begin{split}
        (\frac{\pl}{\pl t} - \Dt_\mu)|\dd u|_\sm^2 &= 2\sum_{k,\ell}\sm^{k\ell}u_{;\ell}\frac{\pl}{\pl t}u_{;k} - \tr_\mu(\Hess_\sm|\dd u|_\sm^2) \\
        &\qquad + (\tr_\mu(\Hess_\sm|\dd u|_\sm^2) - \Dt_\mu|\dd u|_\sm^2) ~.
    \end{split} \label{evo_eqn_du_0} \end{align}
    Due to \eqref{du_in_t} and Corollary \ref{cor_d_theta_asmp} (i) ,
    \begin{align*}
        2\sum_{k,\ell}\sm^{k\ell}u_{;\ell}\frac{\pl}{\pl t}u_{;k} &= 2\sum_{k,\ell}\sm^{k\ell}u_{;\ell}(\ta + \kp u)_{;k} \\
        &\leq 2\sum_{i,j,k,\ell} \mu^{ij}\sm^{k\ell}u_{;\ell}u_{;jki} + c'|\dd u|_\sm(|D^2u|_\sm + |\dd u|_\sm) ~.
    \end{align*}
    By Lemma \ref{lem_muij} and $u_{;jki} - u_{;kij} = u_{;kji} - u_{;kij} = (\text{curvature of }\sm)*\dd u$,
    \begin{align*}
        - \tr_\mu(\Hess_\sm|\dd u|_\sm^2) &= -2\sum_{i,j,k,\ell}\mu^{ij}\sm^{k\ell}(u_{;ki}u_{\ell j} + u_{;kij}u_{;\ell}) \\
        &\leq - |D^2 u|_\sm^2 - 2\sum_{i,j,k,\ell} \mu^{ij}\sm^{k\ell}u_{;\ell}u_{;jki} + c''|\dd u|_\sm^2 ~.
    \end{align*}
    According to Lemma \ref{lem_diffenece_Laplace} and $|D^2u|_\sm \leq \vep_0$,
    \begin{align*}
        \tr_\mu(\Hess_\sm|\dd u|_\sm^2) - \Dt_\mu|\dd u|_\sm^2 &\leq c'''|\dd u|_\sm(|D^3u|_\sm + |D^2u|_\sm + |\dd u|_\sm) ~.
    \end{align*}
    
    Since $|\dd u|_\sm|D^2u|_\sm \leq \dt|D^2u|_\sm^2 + |\dd u|_\sm^2/\dt$ for any $\dt>0$, putting these estimates into \eqref{evo_eqn_du_0} finishes the proof of \eqref{evo_eqn_du}.

    {\textit{The equation \eqref{evo_eqn_ddu}}}.  Differentiating \eqref{du_in_t} gives
    \begin{align} \label{ddu_in_t}
        \frac{\pl}{\pl t}u_{;kp} &= (\ta + \kp u)_{;kp} ~.
    \end{align}
    Similar to \eqref{evo_eqn_du_0},
    \begin{align} \begin{split}
        (\frac{\pl}{\pl t} - \Dt_\mu)|D^2 u|_\sm^2 &= 2\sum_{k,\ell,p,q}\sm^{k\ell}\sm^{pq}u_{;\ell q}\frac{\pl}{\pl t}u_{;kp} - \tr_\mu(\Hess_\sm|D^2 u|_\sm^2) \\
        &\qquad + (\tr_\mu(\Hess_\sm|D^2 u|_\sm^2) - \Dt_\mu|D^2 u|_\sm^2) ~.
    \end{split} \label{evo_eqn_ddu_0} \end{align}
    Due to \eqref{ddu_in_t} and Corollary \ref{cor_d_theta_asmp} (ii) ,
    \begin{align*}
        &\quad 2\sum_{k,\ell,p,q}\sm^{k\ell}\sm^{pq}u_{;\ell q}\frac{\pl}{\pl t}u_{;kp} \\
        &\leq 2\sum_{i,j,k,\ell} \mu^{ij}\sm^{k\ell}\sm^{pq}u_{;\ell q}u_{;jkip} + c'|D^2 u|_\sm (|D^3u|^2_\sm + |D^3u|_\sm + |D^2u|_\sm + |\dd u|_\sm) ~.
    \end{align*}
    By Lemma \ref{lem_muij} and commuting covariant derivatives,
    \begin{align*}
        - \tr_\mu(\Hess_\sm|D^2 u|_\sm^2) &= -2\sum_{i,j,k,\ell,p,q}\mu^{ij}\sm^{k\ell}\sm^{pq}(u_{;kpi}u_{\ell qj} + u_{;kpij}u_{;\ell q}) \\
        &\leq - |D^3 u|_\sm^2 - 2\sum_{i,j,k,\ell} \mu^{ij}\sm^{k\ell}\sm^{pq}u_{;\ell q}u_{;jkip} + c''|D^2 u|_\sm(|D^2 u|_\sm + |\dd u|_\sm) ~.
    \end{align*}
    According to Lemma \ref{lem_diffenece_Laplace},
    \begin{align*}
        \tr_\mu(\Hess_\sm|D^2 u|_\sm^2) - \Dt_\mu|D^2 u|_\sm^2 &\leq c'''|D^3 u|_\sm|D^2 u|_\sm(|D^3u|_\sm + |D^2u|_\sm + |\dd u|_\sm) ~.
    \end{align*}
    It is not hard to see that \eqref{evo_eqn_ddu} follows from these estimates and \eqref{evo_eqn_ddu_0}.
\end{proof}

By combining these quantities, one can produce a monotone quantity along the flow.
\begin{prop} \label{prop_test_psi}
    When the K\"ahler--Einstein constant $\kp\leq0$, there exist positive constants $C_0, C_1\geq1$ and $\vep_1\leq\vep_0$ such that the following holds.  Suppose that $u:L\times[0,T)\to\BR$ solves \eqref{LMCF_potential} with $C(t)\equiv0$.  Let
    \begin{align} \label{test_psi}
        \psi = C_0 u^2 + C_1 |\dd u|_\sm^2 + |D^2u|_\sm^2 ~.
    \end{align}
    If $\psi < (\vep_1)^2$ at $t = 0$, then $\max_{L\times\{t\}}\psi$ is non-increasing in $t$.  Moreover, $\psi$ satisfies
    \begin{align} \label{eqn_psi}
        (\frac{\pl}{\pl t} - \Dt_\mu)\psi &\leq 2C_0\kp u^2 - \frac{1}{4}(C_0|\dd u|_\sm^2 + C_1 |D^2u|_\sm^2 + |D^3u|_\sm^2) ~.
    \end{align}
\end{prop}
\begin{proof}
    The key step is to show that the constants $C_0,C_1,\vep_1$ can be chosen so that \eqref{eqn_psi} holds true on $L\times[0,T')$ as long as $\psi < (\vep_1)^2$ on $L\times[0,T')$.  For $C_1\geq1$ and $\vep_1\leq\vep_0$, $\psi < (\vep_1)^2$ implies that $|\dd u|_\sm<\vep_0$ and $|D^2u|_\sm<\vep_0$.
    
    If $\psi < (\vep_1)^2$, $|u|< \vep_1/\sqrt{C_0}$, and \eqref{evo_eqn_u} implies that
    \begin{align*}
        (\frac{\pl}{\pl t} - \Dt_\mu)C_0u^2 &\leq -C_0|\dd u|_\sm^2 + 2C_0\kp\,u^2 + c\sqrt{C_0}\vep_1(|D^3u|_\sm^2 + |D^2u|_\sm^2 + |\dd u|_\sm^2) ~.
    \end{align*}
    Since $|D^3u|_\sm\cdot|\dd u|_\sm \leq \dt|D^3u|_\sm^2 + \dt^{-1}|\dd u|_\sm^2$ for any $\dt>0$, \eqref{evo_eqn_du} leads to
    \begin{align*}
        (\frac{\pl}{\pl t} - \Dt_\mu)C_1|\dd u|_\sm^2 &\leq -\oh C_1|D^2 u|_\sm^2 + cC_1(\dt|D^3u|_\sm^2 + \dt^{-1}|\dd u|_\sm^2 + |\dd u|_\sm^2) ~.
    \end{align*}
    If $\psi < (\vep_1)^2$, $|D^2u|_\sm< \vep_1$, and \eqref{evo_eqn_ddu} implies that
    \begin{align*}
        (\frac{\pl}{\pl t} - \Dt_\mu)|D^2 u|_\sm^2 &\leq -\oh|D^3u|_\sm^2 + c({\vep_1}|D^3u|_\sm^2 + |D^2u|_\sm^2 + |\dd u|_\sm^2) ~.
    \end{align*}
    It follows that
    \begin{align*}
        (\frac{\pl}{\pl t} - \Dt_\mu)\psi &\leq 2C_0\kp u^2 - \left[ C_0 - c\left(\sqrt{C_0}\vep_1 + C_1(1+\dt^{-1}) + 1\right)\right] |\dd u|_\sm^2 \\
        &\quad - \left[ \oh C_1 - c\left(\sqrt{C_0}\vep_1 + 1\right) \right] |D^2u|_\sm^2 - \left[\oh - c\left(\sqrt{C_0}\vep_1 + C_1\dt + {\vep_1}\right)\right] |D^3u|_\sm^2 ~.
    \end{align*}
    First, choose $C_1 > 10c$; second, choose $\dt < 1/(10cC_1)$; next, choose $C_0 > 10c(C_1(1+\dt^{-1}) + 1)$.  Finally, \eqref{eqn_psi} can be obtained by choosing sufficiently small $\vep_1$.

    To finish the proof of this proposition, consider
    \begin{align*}
        T^* &= \sup\{ T'\in(0,T) : \psi < (\vep_1)^2 \text{ on }L\times[0,T')\} ~.
    \end{align*}
    By a maximal principle argument, $T^*$ must be equal to $T$.
\end{proof}

The last gadget needed in the proof of the main theorem is the evolution equation of the third order derivative of $u$, which is more or less equivalent to the second fundamental form of the graph of $\dd u$. 
 The evolution equation of the second fundamental form along the mean curvature flow was derived in \cite{Wang01}*{section 7}.
 With the help of Corollary \ref{cor_d_theta_asmp} (iii), the proof of the following lemma is basically the same as that of Lemma \ref{lem_evo_eqn_du}, and is left as an exercise for the reader.

\begin{lem} \label{lem_evo_eqn_dddu}
    There exists a constant $c>0$ with the following property.  Suppose that $u:L\times[0,T)\to\BR$ solves \eqref{LMCF_potential}, and $|\dd u|_\sm\leq\vep_0$, $|D^2 u|_\sm\leq\vep_0$ for all $t\in[0,T)$.  Then,
    \begin{align}
        (\frac{\pl}{\pl t} - \Dt_\mu)|D^3 u|_\sm^2 &\leq -\oh|D^4u|_\sm^2 + c(|D^3u|_\sm^4 + |D^3u|_\sm^2 + |D^2u|_\sm^2 + |\dd u|_\sm^2) ~. \label{evo_eqn_dddu}
    \end{align}
\end{lem}

\section{Proof of the Main Theorem}

We are now ready to prove the dynamical stability theorem.
\begin{thm} \label{thm_main}
    Under Assumption \ref{assume_basic}, when the K\"ahler--Einstein constant $\kp\leq0$, there exists a positive constant $\vep>0$ with the following significance.  Suppose that $u_0:L\to\BR$ is a smooth function satisfying $(u_0)^2+|\dd u_0|_\sm^2 + |D^2 u_0|_\sm^2 < \vep^2$.
    Then, the Lagrangian mean curvature flow starting from $\vph\circ\dd u_0:L\to M$ exists for all time.  Moreover, as $t\to\infty$, the flow converges smoothly to $L\subset M$.
\end{thm}
\begin{proof}
Let $C_0,C_1$ and $\vep_1$ be the constants given by Proposition \ref{prop_test_psi}.  Suppose that the initial condition satisfies
\begin{align*}
    C_0(u_0)^2 + C_1|\dd u_0|_\sm^2 + |D^2u_0|_\sm^2 < (\vep_2)^2
\end{align*}
for some $\vep_2\leq\vep_1$; the precise value of $\vep_2$ will be determined later.

Let $T$ be the maximal existence time of the smooth solution $u$ to \eqref{LMCF_potential} (with $C(t)\equiv0$) with initial condition $u_0$.  According to Proposition \ref{prop_test_psi}, $\psi$ defined by \eqref{test_psi} is always less than $(\vep_2)^2$.

{\textit{Uniform boundedness of $D^3u$}}.
With the equation of $\psi$, we employ the trick in \cite{Wang02}*{section 4} to bound $D^3u$.
For any $K>0$, consider the function $\log(1+|D^3 u|^2_\sm) + K\psi$.
It follows from \eqref{evo_eqn_dddu} that
\begin{align*}
    &\quad (\frac{\pl}{\pl t} - \Dt_\mu)\log(1+|D^3 u|^2_\sm) \\
    &\leq \frac{1}{1+|D^3 u|_\sm^2}\left[ -\oh|D^4u|_\sm^2 + c_1(|D^3u|_\sm^4 + |D^3u|_\sm^2 + |D^2u|_\sm^2 + |\dd u|_\sm^2) \right] \\
    &\quad + \ip{\dd\log(1+|D^3u|_\sm^2)}{\dd\left(\log(1+|D^3 u|^2_\sm) + K\psi\right)}_\mu - K\ip{\dd\log(1+|D^3u|_\sm^2)}{\dd\psi}_\mu
\end{align*}
where $\ip{\dd f}{\dd \td{f}}_\mu = \sum_{i,j}\mu^{ij}f_{;i}\td{f}_{;j}$.  By Lemma \ref{muij} and the fact that $\psi < (\vep_2)^2$,
\begin{align*}
    &\quad K\left| \ip{\dd\log(1+|D^3u|_\sm^2)}{\dd\psi}_\mu \right| \\
    &\leq \frac{Kc_2}{1+|D^3u|_\sm^2}|D^3u|_\sm\,|D^4u|_\sm\left(|D^3u|_\sm\,|D^2u|_\sm + |D^2u|_\sm\,|\dd u|_\sm + |\dd u|_\sm\,|u|\right) \\
    &\leq \frac{Kc_3}{1+|D^3u|_\sm^2} \left( \vep_2|D^3u|_\sm^2\,|D^4u|_\sm + (\vep_2)^2|D^3u|_\sm\,|D^4u|_\sm \right) \\
    &\leq \frac{1}{1+|D^3u|_\sm^2} \left( \frac{1}{4}|D^4u|_\sm^2 + c_4K^2(\vep_2)^2(|D^3u|_\sm^4 + |D^3u|_\sm^2) \right)~.
\end{align*}

These estimates with \eqref{eqn_psi} imply that
\begin{align*}
    &\quad (\frac{\pl}{\pl t} - \Dt_\mu)(\log(1+|D^3 u|^2_\sm) + K\psi) \\
    &\leq  (c_1 + c_4 (K\vep_2)^2) |D^3u|_\sm^2 + c_1(|D^2u|_\sm^2 + |\dd u|_\sm^2) - \frac{1}{4}\frac{|D^4u|_\sm^2}{1+|D^3 u|_\sm^2} \\
    &\quad - \frac{K}{4} (|D^3u|_\sm^2 + |D^2u|_\sm^2 + |\dd u|_\sm^2) \\
    &\quad + \ip{\dd\log(1+|D^3u|_\sm^2)}{\dd\left(\log(1+|D^3 u|^2_\sm) + K\psi\right)}_\mu ~.
\end{align*}
By picking $K > 8c_1$ and $\vep_2 < \sqrt{\frac{c_1}{c_4}}\frac{1}{K}$, one finds that $\max_{L\times\{t\}}(\log(1+|D^3 u|^2_\sm) + K\psi)$ is non-increasing in $t$.  It follows that $|D^3u|_\sm$ is uniformly bounded on $L\times[0,T)$.

Since $|\dd u|_\sm$, $|D^2 u|_\sm$ and $|D^3u|_\sm$ are all uniformly bounded, the second fundamental form of the graph of $\dd u$ is uniformly bounded.  It follows from the work of Huisken \cite{Huisken90} that the mean curvature flow exists for all time, $T = \infty$.

Moreover, the standard theory of mean curvature flow implies that all the higher order derivatives of the second fundamental form are uniformly bounded; see for instance \cite{Baker10}*{Proposition 4.8}.  In the current setting, it is equivalent to the boundedness of $\sup_{L\times[0,\infty)}|D^{\ell}u|_\sm$ for any $\ell\geq0$.

{\textit{Convergence}}.  With the boundedness of $|D^{\ell}u|_\sm$, there exists a time sequence $t_j\to\infty$ such that $u(\,\cdot\,,t_j)$ converges smoothly to $u_\infty$ as $j\to\infty$.  Thus, the graph of $\dd u(\,\cdot\,,t_j)$ converges smoothly to the graph of $\dd u_\infty$.

It is known that a K\"ahler--Einstein metric is analytic; see \cites{Morrey1958, DH1988}.  Note that the mean curvature flow is the negative gradient flow of the volume functional.  With the analyticity of the metric, the general theorem of Simon \cite{Simon83}*{Corollary 2} implies that the graph of $\dd u(\,\cdot\,t)$ converges smoothly to the graph of $\dd u_\infty$ as $t\to\infty$.

{\textit{The limit}}.  When $\kp < 0$, it follows from \eqref{eqn_psi} that $\psi$ decays to $0$ exponentially, and thus $u_\infty \equiv 0$.

When $\kp = 0$, it follows from \eqref{Weinstein_cpx_volumeform} that $\Om_\vph$ is holomorphic volume form of constant magnitude, and $U\subset T^*L$ is a Calabi--Yau manifold.  The graph of $\dd u_\infty$ is a minimal Lagrangian.  The function $u_\infty$ can be viewed as a function on the ambient space $U\subset T^*L$.  It is not hard to see that it gives a time-independent Hamiltonian isotopy between $(\dd u_\infty)(L)$ and $L$.  By \cite{ThomasYau02}*{Lemma 4.2}, $(\dd u_\infty)(L)$ coincides with $L$, and $u_\infty$ is a constant function.  It finishes the proof of this theorem.
\end{proof}

We remark that the Thomas--Yau uniqueness theorem invoked in the last step has further generalizations; especially by Imagi \cite{Imagi23} and Li \cite{YLi22}.

\begin{rmk}
    When $\kp = 0$, one may also apply Proposition \ref{prop_test_psi} to conclude the limit is $L$.  Replace $u$ by $\td{u} = u-\int_L u\dd V_L/\int_L 1\dd V_L$, where $\dd V_L$ is the volume form of $(L,\sm)$.  It follows that $\int_L \td{u}\,\Dt_\sm \td{u}\,\dd V_L \leq -\ld_1\,\int_L\td{u}^2\,\dd V_L$.  Note that $\Dt_\sm\td{u} - \Dt_\mu\td{u} = \Dt_\sm{u} - \Dt_\mu{u}$ can be bounded by Lemma \ref{lem_muij}.  By integrating \eqref{eqn_psi} over $L$, one can show that $\int_L\td{u}^2\dd V_L$ decays to zero exponentially, provided that $\vep_2$ is sufficiently small.
\end{rmk}

\subsection{Positive K\"ahler--Einstein constant}

We finish this paper by mentioning the case when $\kp > 0$.  In this case, the formula \eqref{Oh_formula} of Oh \cites{Oh90, Oh93} implies that a minimal Lagrangian is Hamiltonian stable if and only if the first eigenvalue of $\Dt_L$ is no less than $\kp$.  Oh also showed that $\BR\BP^n$ and the Clifford torus are Hamiltonian stable in $\BC\BP^n$, and their first eigenfunctions do give \emph{integrable} variations.  That is to say, for any first eigenfunction $\psi$, there exists a one-parameter family of Hamiltonian isotopic minimal Lagrangians $\{L_s\}_{-\epsilon<s<\epsilon}$ with $L_0 = L$, and $\om(\frac{\pl L_s}{\pl s}|_{s=0},\,\cdot\,)|_L$ is $\dd\psi$.

In the situation of Hamiltonian stable minimal Lagrangian whose first eigenfunctions are integrable, one can apply the result of Simon \cite{Simon83}*{Theorem 2} to conclude the $C^3$ (in potential $u_0$) dynamical stability.  Note that the limit may be a different minimal Lagrangian from the original one.  For instance, if one perturbs the equator of the round $2$-sphere so that it divides the sphere into two components of the same area, the curve will remain so along the mean curvature flow.  Thus, it will not shrink to a point, but it may converge to a different equator as $t\to\infty$.

It would be interesting to see whether the dynamical stability can be improved to $C^2$-closeness in potential when $\kp>0$.



\begin{bibdiv}
\begin{biblist}

\bib{Baker10}{book}{
   author={Baker, Charles},
   title={The mean curvature flow of submanifolds of high codimension},
   note={Thesis (Ph.D.)--Australian National University},
   date={2010},
}

\bib{DH1988}{article}{
   author={Delano\"{e}, Ph.},
   author={Hirschowitz, A.},
   title={About the proofs of Calabi's conjectures on compact K\"{a}hler manifolds},
   journal={Enseign. Math. (2)},
   volume={34},
   date={1988},
   number={1-2},
   pages={107--122},
}

\bib{HarveyLawson82}{article}{
   author={Harvey, Reese},
   author={Lawson, H. Blaine, Jr.},
   title={Calibrated geometries},
   journal={Acta Math.},
   volume={148},
   date={1982},
   pages={47--157},
}




\bib{Huisken90}{article}{
   author={Huisken, Gerhard},
   title={Asymptotic behavior for singularities of the mean curvature flow},
   journal={J. Differential Geom.},
   volume={31},
   date={1990},
   number={1},
   pages={285--299},
}

\bib{Imagi23}{article}{
   author={Imagi, Yohsuke},
   title={Generalized Thomas-Yau uniqueness theorems},
   conference={
      title={Birational geometry, K\"{a}hler-Einstein metrics and degenerations},
   },
   book={
      series={Springer Proc. Math. Stat.},
      volume={409},
      publisher={Springer},
   },
   date={2023}, 
   pages={323--342},
}

\bib{JL23}{article}{
   author={Jin, Xishen},
   author={Liu, Jiawei},
   title={Stability of the generalized Lagrangian mean curvature flow in the cotangent bundle},
   journal={},
   volume={},
   date={},
   number={},
   pages={},
   status={preprint},
}

\bib{YLi22}{article}{
   author={Li, Yang},
   title={Quantitative Thomas-Yau uniqueness},
   journal={},
   volume={},
   date={},
   number={},
   pages={},
   eprint={arXiv:2207.08047},
   status={preprint},
}

\bib{LotaySchulze20}{article}{
   author={Lotay, Jason D.},
   author={Schulze, Felix},
   title={Consequences of strong stability of minimal submanifolds},
   journal={Int. Math. Res. Not. IMRN},
   date={2020},
   number={8},
   pages={2352--2360},
}


\bib{McDuffSalamon17}{book}{
   author={McDuff, Dusa},
   author={Salamon, Dietmar},
   title={Introduction to symplectic topology},
   series={Oxford Graduate Texts in Mathematics},
   edition={3},
   publisher={Oxford University Press}, 
   date={2017},
   pages={xi+623},
}


\bib{Morrey1958}{article}{
   author={Morrey, Charles B., Jr.},
   title={On the analyticity of the solutions of analytic non-linear elliptic systems of partial differential equations. I. Analyticity in the interior},
   journal={Amer. J. Math.},
   volume={80},
   date={1958},
   pages={198--218},
}

\bib{Kodaira06}{book}{
   author={Morrow, James},
   author={Kodaira, Kunihiko},
   title={Complex manifolds},
   note={Reprint of the 1971 edition with errata},
   publisher={AMS Chelsea Publishing}, 
   date={2006},
   pages={x+194},
}

\bib{Neves13}{article}{
   author={Neves, Andr\'{e}},
   title={Finite time singularities for Lagrangian mean curvature flow},
   journal={Ann. of Math. (2)},
   volume={177},
   date={2013},
   number={3},
   pages={1029--1076},
}

\bib{Oh90}{article}{
   author={Oh, Yong-Geun},
   title={Second variation and stabilities of minimal Lagrangian submanifolds in K\"{a}hler manifolds},
   journal={Invent. Math.},
   volume={101},
   date={1990},
   number={2},
   pages={501--519},
}

\bib{Oh93}{article}{
   author={Oh, Yong-Geun},
   title={Volume minimization of Lagrangian submanifolds under Hamiltonian
   deformations},
   journal={Math. Z.},
   volume={212},
   date={1993},
   number={2},
   pages={175--192},
}

\bib{Oh94}{article}{
   author={Oh, Yong-Geun},
   title={Mean curvature vector and symplectic topology of Lagrangian submanifolds in Einstein-K\"{a}hler manifolds},
   journal={Math. Z.},
   volume={216},
   date={1994},
   number={3},
   pages={471--482},
}


\bib{Simon83}{article}{
   author={Simon, Leon},
   title={Asymptotics for a class of nonlinear evolution equations, with applications to geometric problems},
   journal={Ann. of Math. (2)},
   volume={118},
   date={1983},
   number={3},
   pages={525--571},
}

\bib{Smoczyk96}{article}{
   author={Smoczyk, Knut},
   title={A canonical way to deform a Lagrangian submanifold},
   journal={},
   volume={},
   date={},
   number={},
   pages={},
   eprint={arXiv:dg-ga/9605005},
   status={preprint},
}


\bib{SmoczykTsuiWang19}{article}{
   author={Smoczyk, Knut},
   author={Tsui, Mao-Pei},
   author={Wang, Mu-Tao},
   title={Generalized Lagrangian mean curvature flows: the cotangent bundle case},
   journal={J. Reine Angew. Math.},
   volume={750},
   date={2019},
   pages={97--121},
}

\bib{SmoczykWang11}{article}{
   author={Smoczyk, Knut},
   author={Wang, Mu-Tao},
   title={Generalized Lagrangian mean curvature flows in symplectic manifolds},
   journal={Asian J. Math.},
   volume={15},
   date={2011},
   number={1},
   pages={129--140},
}


\bib{ThomasYau02}{article}{
   author={Thomas, R. P.},
   author={Yau, S.-T.},
   title={Special Lagrangians, stable bundles and mean curvature flow},
   journal={Comm. Anal. Geom.},
   volume={10},
   date={2002},
   number={5},
   pages={1075--1113},
}

\bib{TsaiWang20}{article}{
   author={Tsai, Chung-Jun},
   author={Wang, Mu-Tao},
   title={A strong stability condition on minimal submanifolds and its implications},
   journal={J. Reine Angew. Math.},
   volume={764},
   date={2020},
   pages={111--156},
}

\bib{Wang01}{article}{
   author={Wang, Mu-Tao},
   title={Mean curvature flow of surfaces in Einstein four-manifolds},
   journal={J. Differential Geom.},
   volume={57},
   date={2001},
   number={2},
   pages={301--338},
}

\bib{Wang02}{article}{
   author={Wang, Mu-Tao},
   title={Long-time existence and convergence of graphic mean curvature flow in arbitrary codimension},
   journal={Invent. Math.},
   volume={148},
   date={2002},
   number={3},
   pages={525--543},
}

\bib{Uhlenbeck82}{article}{
   author={Uhlenbeck, Karen K.},
   title={Removable singularities in Yang-Mills fields},
   journal={Comm. Math. Phys.},
   volume={83},
   date={1982},
   number={1},
   pages={11--29},
}

\end{biblist}
\end{bibdiv}
\end{document}